\documentclass[a4paper,reqno]{amsart}
\usepackage[utf8]{inputenc}
\usepackage{amssymb}
\usepackage{amsmath}
\usepackage{mathtools}
\usepackage{ifpdf}
\usepackage{vmargin}

\usepackage{graphicx}  
\usepackage{subfigure}
\usepackage{epstopdf}
\usepackage{caption}

\usepackage{setspace}
\usepackage{xcolor}

\ifpdf 
\usepackage[pdftex]{hyperref}
\else
\usepackage[hypertex]{hyperref}
\fi

\hypersetup{
pdftitle={ },
pdfauthor={},
}

\newtheorem{theorem}{Theorem}
\newtheorem{proposition}[theorem]{Proposition}
\newtheorem{lemma}[theorem]{Lemma}
\newtheorem{remark}[theorem]{Remark}

\newtheorem{definition}[theorem]{Definition}

\numberwithin{equation}{section}

\newcommand{\R}{\mathbb{R}}
\newcommand{\C}{\mathbb C}
\newcommand{\Z}{\mathbb Z}

\newcommand{\bee}{\begin{eqnarray*}}
\newcommand{\eee}{\end{eqnarray*}}

\numberwithin{equation}{section}
\numberwithin{theorem}{section}

\begin{document}
\onehalfspacing

\title[Global well-posedness of the defocusing,  cubic NLW outside of the ball]{Global well-posedness of the defocusing,  cubic nonlinear wave equation outside of the ball with radial data }

\date\today

\author[Guixiang Xu]{Guixiang Xu}
\address{\hskip-1.15em Guixiang Xu 
	\hfill\newline Laboratory of Mathematics and Complex Systems,
	\hfill\newline Ministry of Education,
	\hfill\newline School of Mathematical Sciences,
	\hfill\newline Beijing Normal University,
	\hfill\newline Beijing, 100875, People's Republic of China.}
\email{guixiang@bnu.edu.cn}

\author[Pengxuan Yang]{Pengxuan Yang}
\address{\hskip-1.15em Pengxuan Yang 
	\hfill\newline Laboratory of Mathematics and Complex Systems,
	\hfill\newline Ministry of Education,
	\hfill\newline School of Mathematical Sciences,
	\hfill\newline Beijing Normal University,
	\hfill\newline Beijing, 100875, People's Republic of China.}
\email{1413636745@qq.com}



\begin{abstract}
We consider the defocusing, cubic nonlinear wave equation with zero Dirichlet boundary value in the exterior domain $\Omega = \R^3\backslash \overline{ B(0,1)}$. We make use of the distorted Fourier transform in  \cite{Horm:PDE IV, LiSZ:NLS, LiXZ:NLS, Sch:LW-P theory, Taylor:PDE:II} to establish the dispersive estimate and the global-in-time (endpoint) Strichartz estimate of the linear wave equation outside of the unit ball with radial data. As an application, we combine the Fourier truncation method as those in \cite{Bourgain98:FTM, GallPlan03:NLW, KenigPV00:NLW} with the energy method to show global well-posedness of radial solution to the defoucusing, cubic nonlinear wave equation outside of the unit ball in the Sobolev space $\left(\dot H^{s}_{D}(\Omega) \cap L^4(\Omega) \right)\times \dot H^{s-1}_{D}(\Omega)$ for $s>3/4$.  To the best of the authors' knowledge, it is first low regularity global well-posedness of semilinear wave equation with zero Dirichlet boundary value in the exterior domain.
\end{abstract}

\maketitle
 


\section{Introduction}
In this paper, we consider the  3D defocusing, cubic wave equation with zero Dirichlet boundary value outside of the unit ball
\begin{equation}\label{eq:nlu}
\begin{cases}
\partial^2_t u -\Delta u + u^3=0, & (t,x)\in I \times\Omega, \\
u(0,x)=u_0(x),  & x\in \Omega, \\
\partial_t u(0, x)= u_1(x), & x\in \Omega, \\
u(t,x)=0, & x\in \partial \Omega, 
\end{cases}
\end{equation}
where $0\in I$, the exterior domain $\Omega = \R^3\backslash \overline{ B(0,1)}$, the function $u: I\times \Omega \longrightarrow \R$ and initial data $(u_0, u_1)$ is radial and belongs to the Sobolev space $\dot H^{s}_{D}(\Omega)\times \dot H^{s-1}_{D}(\Omega)$. The solution of equation \eqref{eq:nlu} formly enjoys the energy conservation 
\begin{align}
E(u)(t) : = &\; \int_{\Omega} \left( \frac{1}{2} |\nabla u(t,x)|^2 + \frac12 |\partial_t u(t,x)|^2 +\frac14 |u(t,x)|^4 \right)\; dx \nonumber\\
=&\;  E(u)(0), \quad \forall\; t\in I.\label{def:energy}
\end{align}
Typically, we call \eqref{eq:nlu} the  (conformal) $\dot H^{1/2}$-critical wave equation by the scaling analysis.

There are lots of works focusing on the Cauchy problem of semilinear wave equation including the energy (sub)-critical wave equation (see  \cite{BahChe:NLW,Dod:NLW:sct1, Dod:NLW:sct2,  DuyKM:Odd NLW:SRC, JendrejL:NLW:SRC, Kenig:book, LindSog95:NLW, Roy:NLW:radial,Roy:NLW, ShatS:NLW, Sogge95:book} and reference therein),  and there are also many important results about the energy-critical wave/Schr\"odinger equations outside of the domain (see \cite{BurqLP:NLW, BurqP:def NLW: Neum prob,  DuyLaf:NLW:Neum prob, DuyYang24:NLW, KVZ16:NLS:Scat, LiSZ:NLS, LiXZ:NLS, SmithSogge95:NLW JAMS} and reference therein).  The energy conservation law plays  a crutial rule in long time dynamics in both of the Cauchy problem (IVP) and the initial boundary value problem (IBVP) of the energy-critical wave/Schr\"odinger equations.

As for the Sobolev space in the exterior domain, Killip, Visan and Zhang systematically establish the boundedness of Riesz transform,  the Littlewood-Paley theory of the Sobolev space on a convex obstacle and the equivalence between $\dot H^{s,p}_{D}(\Omega)$ with $\dot H^{s,p}(\R^d)$ for the sharp ranges of  $s$ and $p$ in \cite{KVZ16:IMRN Riesz tf} by use of the heat kernel estimate in the exterior of a convex obstacle in \cite{Zhang:Heat kernel} and classical harmonic analysis theory including the Calder\'on-Zygnumd theory and Mikhlin's multiplier theorem in \cite{Gr14:CFA, MuscaluS:book1, Stein:book:SI}, etc, (see also \cite{IvPlan17:CPDE: heat flow, LiSZ:NLS}). We will recall more useful properties of the Sobolev space $\dot H^{s,p}_{D}(\Omega)$ in \cite{KVZ16:IMRN Riesz tf, LiSZ:NLS, LiXZ:NLS} in Section \ref{sect:DF transf}.

 We now give the definition of local-well-posedness of \eqref{eq:nlu} in  $\dot H^{s}_{D}(\Omega)\times \dot H^{s-1}_{D}(\Omega)$.

\begin{definition}
The equation \eqref{eq:nlu}  is said to be locally well-posed in $\dot H^{s}_{D}(\Omega)\times \dot H^{s-1}_{D}(\Omega)$  if there exists an open interval $I\subset \R$ containing $0$ such that
$(1)$ there is a unique solution in $\dot H^{s}_{D}(\Omega)\times \dot H^{s-1}_{D}(\Omega)$;
$(2)$ the solution is continuous in time, that is $(u, \partial_t u)\in C\left(I; \; \dot H^s_{D}(\Omega) \times \dot H^{s-1}_{D}(\Omega)\right)$;
$(3)$ the solution depends continuously on the initial data. 
\end{definition}

By the Strichartz estimate in Theorem \ref{thm:St est} and standard Picard fixed point argument, we can obtain the following local well-posedness result of \eqref{eq:nlu} in $\dot H^s_{D}(\Omega) \times \dot H^{s-1}_{D}(\Omega)$  with $\frac12\leq s< \frac32$.
\begin{theorem}\label{thm:nlu lwp}
	The equation \eqref{eq:nlu} is locally well-posed in  $\dot H^{1/2}_{D}(\Omega)\times \dot H^{-1/2}_{D}(\Omega)$ on some interval $I=\left(-T, T\right)$. Moreover, the regularity of initial data is enough to give a lower bound on the time of well-posedness, that is, there exists some positive lifespan $T = T\big(\|(u_0, u_1)\|_{\dot H^{s}_{D}(\Omega) \times \dot H^{s-1}_{D}(\Omega)}\big)$ for any $\frac12 <s<\frac32$.
\end{theorem}

The local well-posedness theory together with the energy conservation law implies that

\begin{theorem}
		The equation \eqref{eq:nlu} is globally well-posed in  in $\left(\dot H^1_{D}(\Omega) \cap L^4(\Omega)\right) \times   L^2(\Omega)$.
	
\end{theorem}


As shown in the local/global well-posedness theory in \cite{BurqP:def NLW: Neum prob, DuyLaf:NLW:Neum prob, SmithSogge95:NLW JAMS}, the (local-in-time) Strichartz estimate is one of useful estimates in long time behavior of the wave equations. In fact, the dispersive estimate and the (local-in-time) Strichartz estimate  for the wave/Schr\"odinger equations in the (exterior) domain themselves are extremely complicated, we can refer to \cite{BlairSS:NLS:St, Burq03: NLW CPDE, BurqLP:NLW, DuyYang24:LW Str, HidanoMSSZ:AST est, Iv:NLS St est, IvLasLebP:LW:dis est, IvLebeauP:LW:dis est,  SmithSogge95:NLW JAMS, SmithSogge00:NLW CPDE, SmithSW:AST est, StaT:NLS St est}  and reference therein. 

In this paper, we consider the Dirichlet boundary value problem \eqref{eq:nlu} with radial data outside of the unit ball, there is a crucial analysis tool, that is, the explicit distorted Fourier transform adapted to the Dirichlet-Laplacian operator $-\Delta_{\Omega}$ under the radial assumption (see also \cite{LiSZ:NLS, LiXZ:NLS, Taylor:PDE:II}), to simplify the proof of the dispersive estimate and the global-in-time (endpoint) Strichartz estimate (see Theorem \ref{thm:St est} and Theorem \ref{thm:endp est}) of linear wave equation  outside of the unit ball with radial data. As for the distorted Fourier transform adapted to the Schr\"odinger operator $-\Delta+V$ for some potential $V$, we can refer to \cite{Agm75:Sch Op, CollG:NLS, DonK:PNLW, KriST:NLW bup, KriST:NLS}, and reference therein. The Schr\"odinger operator $-\Delta+V$ is also discussed in \cite{Horm:PDE IV, Melrose:book}.

As an application of the global-in-time (endpoint) Strichartz estimate and the energy method, we can make use of the Fourier truncation method as those in \cite{Bourgain98:FTM, GallPlan03:NLW, KenigPV00:NLW}  to show global well-posedness of equation \eqref{eq:nlu} with radial data in low regularity Sobolev space $\left(\dot H^{s}_{D}(\Omega) \cap L^4(\Omega) \right) \times \dot H^{s-1}_{D}(\Omega)$ for some $s<1$,  where the solution maybe have infinite energy. Unlike the energy-critical wave/Schr\"odinger problem, there is no a priori conservation law for the rough solution of \eqref{eq:nlu} to ensure that the low regularity Sobolev norm of solution will remain bounded during the lifespan of the solution.

Main result in this paper is the following.

\begin{theorem}\label{thm:nlu gwp} Let $\frac{3}{4}<s<1$, then
	the equation \eqref{eq:nlu} with radial data is globally well-posed in $\left(\dot H^{s}_{D}(\Omega) \cap L^4(\Omega) \right) \times \dot H^{s-1}_{D}(\Omega)$. More precisely,  for arbitrarily large time $T$, the solution $u$ obeys the following estimate
	\begin{align*}
	\left \| (u, \partial_t u) \right\|_{C\left([0, T); \; \dot H^{s}_{D}(\Omega)\right) \times C\left([0, T); \; \dot H^{s-1}_{D}(\Omega)\right)  } \lesssim T^{\frac{3(1-s)(2s-1)}{4s-3}}.
	\end{align*}
\end{theorem}

\begin{remark} For $s\in (\frac12, \frac32)$, to see that $u(t,x)=0$ for $x\in \partial \Omega$, we may use the relation that $\dot H^{s}_{D}(\Omega) =  \dot H^{s}_{00}(\Omega) $ in   Theorem \ref{thm:Sob eq}, where $  \dot H^{s}_{00}(\Omega) $ is defined as the completion of $C^{\infty}_{c}(\Omega)$ in $\dot H^{s,p}(\R^3)$, that is, one must extend the function outside of the unit ball as zero. (See Definition \ref{def:Sob 0 extd}).
	\end{remark}

To the best of the authors' knowledge, it is first low regularity global well-posedness of semilinear wave equation with zero Dirichlet boundary value in the exterior domain.  We will take the strategy from \cite{Bourgain98:FTM, GallPlan03:NLW, KenigPV00:NLW}  to decompose the solution of \eqref{eq:nlu} into small global solution $w$ of \eqref{eq:nlu} with high-frequency-localized initial data and the remaider with low-frequency-localized initial data, which energy is almost conserved for arbitrarily large time $T$ if 
 choosing the  frequency cutoff carefully, which is similar to the well-known almost conservation law initially developed by Colliander, Keel, Staffilani, Takaoka and Tao in \cite{CKSTT:ACL}.

\subsection*{Notation}\label{sect:notation}Throughout the paper, we  use the notation $X \lesssim Y$, or $Y \gtrsim X$ to denote the statement that $X\leq C Y$  for some constant $C$, which may vary from line to line.    We use $X\approx Y$ to denote the statement $X\lesssim Y \lesssim X$.

Lastly, this paper is organized as follows. In Section  \ref{sect:DF transf}, we recall the distorted Fourier transform adapted to the Dirichlet-Laplacian operator $-\Delta_{\Omega}$ in the radial case and the related Littlewood-Paley theory.  In Section \ref{sect:linear est}, we use the distorted Fourier transform to prove the dispersive estimate and the global-in-time (endpoint) Strichartz estimate to linear wave equation outside of the unit ball with radial data, which gives a short proof of the global-in-time Strichartz estimate when the initial data is radial. In Section \ref{sect:gwp}, we combine the Fourier truncation method, the energy method with the (endpoint) Strichartz estimate to show that  the solution to the difference equation with low-frequency-localized initial data has almost conserved energy and can exist for arbitrarily large time $T$,  which can be used to complete the proof of Theorem \ref{thm:nlu gwp} together with the existence of global small solution with high-frequency-localized initial data. In Appendix \ref{app:A}, we show the propogation of the half-wave operator with radial data in the whole space $\R^3$.

\section{Distorted Fourier Transform and Littlewood-Paley theory}\label{sect:DF transf}
In this Section, we consider the 3D Dirichlet-Laplacian operator outside of the unit ball $\Omega=\R^3\backslash \overline{B(0,1)}$ with domain $H^2(\Omega)\cap H^{1}_{0}(\Omega)$, which  we denote by $-\Delta_{\Omega}$. we recall the distorted Fourier transform adapted to the operator $-\Delta_{\Omega}$ in the radial case from \cite{LiSZ:NLS, LiXZ:NLS} and the  Sobolev space $\dot H^{s,p}_{D}(\Omega)$ in \cite{KVZ16:IMRN Riesz tf}. We can also refer to \cite{Agm75:Sch Op,BPSS:invers pot,Horm:PDE IV, KMVZZ:Sch op:18, ReedS:book:IV,Taylor:PDE:II} for spectral properties of Schr\"odinger operator and  to \cite{BahChe:book, Gr14:CFA, MuscaluS:book1, Stein:book:SI} for  classical Fourier analysis and the Littlewood-Paley theory on the whole space.

Note that the Dirichlet-Laplacian operator $-\Delta_{\Omega}$ is positive, self-adjoint operator, 
its spectral theory  is similar as that of the Laplacian operator $-\Delta_{\R^3}$,  we can also refer to \cite{Agm75:Sch Op, DuyYang24:LW Str,KMVZZ:Sch op:18, LaxPhi:book, LiSZ:NLS,LiXZ:NLS,Sch:LW-P theory}, and reference therein.  The essential spectrum of the operator $-\Delta_{\Omega}$ is $[0, \infty)$; The operator $-\Delta_{\Omega}$ has no positive eigenvalues embedded into $(0, \infty)$ and has no negative eigenvalues; Moreover $0$ is not an eigenvalue or a resonance of  the operator $-\Delta_{\Omega}$.

Now we follow the argument in \cite{LiSZ:NLS, LiXZ:NLS} and \cite{Agm75:Sch Op,KMVZZ:Sch op:18, KVZ16:IMRN Riesz tf} to recall the distorted Fourier transform and Sobolev space associated to the Dirichlet-Laplacian operator $-\Delta_{\Omega}$. The spectral resolution for radial functions on $\Omega=\R^3\backslash \overline{B(0,1)}$ is expressed simply by the radial, generalized  eigenfunctions
$$-\Delta\; e_{\lambda} = \lambda^2\; e_{\lambda}$$
for all $\lambda>0$,  which satisfies the Sommerfeld radiation condition (See Chapter $9$ in \cite{Taylor:PDE:II}), namely
%
\begin{equation}\label{quan:eign funt}
e_{\lambda} (r)=\frac{\sin \lambda(r-1)}{r}, \quad r=|x|\geq 1.
\end{equation}
We can refer to  \cite{Agm75:Sch Op, LaxPhi:book, Melrose:book}  (see also \cite{Ike:eigexp}) for more introductions about the generalized eigenfunctions, which behave like the plane waves. We can also refer to \cite{Sch:LW-P theory}  for the Littlewood-Paley theory associated to  the distorted Fourier transform.

For the radial, tempered distributions $f\in \mathcal{S}'(\R^3)$, supported on $\Omega$, we denote the distorted Fourier transformation $\mathcal{F}_Df(\lambda)$ for $\lambda>0$ by 
\begin{align}
\mathcal{F}_D f(\lambda):= &\; \frac{\sqrt{2}}{\sqrt{\pi}} \int^{\infty}_{1} e_{\lambda}(s) f(s)\; s^2 \; ds  \nonumber\\
= &\;  \frac{\sqrt{2}}{\sqrt{\pi}} \int^{\infty}_{1}  \frac{\sin \lambda(s-1)}{s} f(s)\; s^2 \; ds.\label{def:DF}
\end{align}  

Note that the following resolution of identity
\begin{align}
\frac{2}{\pi}\int^{\infty}_{0} e_{\lambda}(r) \, e_{\lambda}(s)\, d\lambda= &\;  \frac{1}{\pi}\int^{\infty}_{-\infty} e_{\lambda}(r) \, e_{\lambda}(s)\, d\lambda  \nonumber \\
= & \; \frac{1}{2\pi r s} \int^{\infty}_{-\infty} \big[ \cos{\lambda(r-s)}  -\cos{\lambda (r+s-2)} \big]\, d\lambda  \nonumber \\
= &\; \frac{\delta(r-s)}{s^2}, \; \text{for}\; r, s>1 \label{Id:resolution}
\end{align}
 from which  it follows that $$\mathcal{F}^{-1}_D\mathcal{F}_D f = f$$
for the radial function $f\in \mathcal{S}(\R^3)$ supported in $\Omega$, where $\mathcal{F}^{-1}_D$ is the formal adjoint, defined on the tempered distributions $g$ as the restriction to $r =|x|\geq 1 $ of 
\begin{align}
\mathcal{F}^{-1}_D g(r)= &\;  \frac{\sqrt{2}}{\sqrt{\pi}} \int^{\infty}_{0} e_{\lambda}(r) g(\lambda)\; d\lambda \nonumber \\ 
= &\;  \frac{\sqrt{2}}{\sqrt{\pi}} \int^{\infty}_{0}  \frac{\sin \lambda(r-1)}{r} g(\lambda)\; d\lambda.\label{def:inverse DF}
\end{align}  

By \eqref{def:DF}, we know that $\mathcal{F}_D f$ is an odd function in $\lambda$ if $f$ is a radial function supported on $\Omega$, then if $g\in \mathcal{S}(\R)$ is an odd function, we have the similar estimate $$\mathcal{F}_D\mathcal{F}^{-1}_D  g = g$$  as that in \eqref{Id:resolution} , More precisely, we have
\begin{align}
\frac{2}{\pi} \int^{\infty}_{1}\int^{\infty}_{0}&  e_{\lambda}(s) e_{\mu}(s) g(\mu) \; d\mu\; s^2\; ds \nonumber \\
 = &\; \frac{1}{\pi} \int^{\infty}_{1}\int^{\infty}_{-\infty} e_{\lambda}(s) e_{\mu}(s) g(\mu) \; d\mu\; s^2\; ds \nonumber \\
 = &\; \frac{1}{2\pi} \int^{\infty}_{-\infty} \int^{\infty}_{1} \big[ \cos{(\lambda-\mu)(s-1)} -\cos{(\lambda+\mu)(s-1)}\big]  ds\; g(\mu) d\mu\;  \nonumber \\
  = &\;  \frac{1}{4\pi}\int^{\infty}_{-\infty}  \int^{\infty}_{-\infty} \big[ \cos{(\lambda-\mu)s} -\cos{(\lambda+\mu)s}\big]  ds\, g(\mu) d\mu\;  \nonumber \\
    = &\;  \frac12 \int^{\infty}_{-\infty}  \big[ \delta(\lambda-\mu) -\delta(\lambda+\mu)\big]\, g(\mu) d\mu 
=\;  g(\lambda), \label{Id:resol2}
\end{align}
where we use the fact that the function $g$ is an odd one in the last equality. This implies for any radial function $f\in C^{\infty}_{c}(\Omega)$ that
\begin{align*}
 \int^{\infty}_{1} |f(s)|^2 \; s^2\; ds
=& \;   \frac{2}{\pi}  \int^{\infty}_{1}  \int^{\infty}_{0} e_{\lambda}(s) \left (\mathcal{F}_D f\right)(\lambda) \; d\lambda \cdot   \int^{\infty}_{0} e_{\mu}(s) \overline{\left(\mathcal{F}_D f\right)(\mu)} \; d\mu \; s^2\; ds \\
=& \;    \int^{\infty}_{0}  \frac{2}{\pi}  \int^{\infty}_{1}  \int^{\infty}_{0}  e_{\lambda}(s) e_{\mu}(s)  \overline{\left(\mathcal{F}_D f\right)(\mu)} \; d\mu\; s^2\; ds \cdot    \left(\mathcal{F}_D f\right)(\lambda) \; d\lambda\;  \\
=& \;      \int^{\infty}_{0}    \left(\mathcal{F}_D f\right)(\lambda) \overline{\left(\mathcal{F}_D f\right)(\lambda)} \; d\lambda \\
= & \;      \int^{\infty}_{0}    \left| \left(\mathcal{F}_D f\right)(\lambda)\right|^2 \; d\lambda. 
\end{align*}
Consequently, $f\longrightarrow \mathcal{F}_D f $ induces an isometric map
\begin{align*}
\mathcal{F}_D : \; L^{2}\left([1, \infty), s^2\, ds\right) \longrightarrow L^{2}\left([0, \infty), \, d\lambda\right).
\end{align*}
It is worth noting that unlike classical Fourier transform on the whole space, the spectral supports are not additive under function multiplication for the above (inverse) distorted Fourier transform.

Given a bounded function $m(\lambda)$, which for convenience we assume to be defined on all of $\R$ and even in $\lambda$, and radial function $f\in C^{\infty}_{c}(\Omega)$, we define
\begin{align*}
m\left(\sqrt{-\Delta_{\Omega}}\right)f (r) =  \mathcal{F}^{-1}_D \big(m(\cdot)\, \mathcal{F}_D  f \big)(r).
\end{align*}
This defines a functional calculus on $L^2_{rad}(\Omega)$ and takes the expression as
\begin{align*}
m(\sqrt{-\Delta_{\Omega}}) f (r) = \int^{\infty}_{1} K_m(r, s) f(s)\, s^2 \, ds
\end{align*}
with 
\begin{align*}
	K_m(r, s) = \frac{2}{\pi}\, \int^{\infty}_{0} e_{\lambda}(r)\, e_{\lambda}(s)\, m(\lambda)\; d\lambda. 
	\end{align*}

In general, we have the following Mikhlin Multiplier theorem.
\begin{theorem}[\cite{KVZ16:IMRN Riesz tf, LiSZ:NLS}]\label{thm:multiplier}
	Suppose $m: [0, \infty) \rightarrow \C$ obeys 
	\begin{align*}
	\left|\partial^{k}_{\lambda} m (\lambda)\right| \lesssim \lambda^{-k}
	\end{align*}
	for all integer $ k \in [0, 2] $. Then $m(\sqrt{-\Delta_{\Omega}}) $, which we define via the $L^2$ functional calculus, extends uniquely from $L^2(\Omega)\cap L^{p}(\Omega)$ to a bounded operator on $L^p(\Omega)$, for all $1<p<\infty$.
	\end{theorem}
\begin{proof}
	In the radial case, we can use the Schur's Lemma and interpolation to obtain the result, please refer to \cite{LiSZ:NLS} for more details, and in general case, we need to use the heat kernel estimate in the exterior domain $\Omega$ in \cite{Zhang:Heat kernel} and classical Calder\'on-Zygmund theory to show the boundedness of the multiplier, please see more details in \cite{KVZ16:IMRN Riesz tf}. 
	\end{proof}
More Mikhlin multiplier results adapted to the Schr\"odinger operator $-\Delta + V$, we can refer to \cite{GermHW14:PNLS nonlinear est, Sch:LW-P theory}.  Based on the above Mikhlin's multiplier theorem, we can describe  basic ingredients of the Littlewood-Paley theory adapted to the Dirichlet-Laplacian operator $-\Delta_{\Omega}$. 
Fix $\phi: [0, \infty) \rightarrow [0, 1]$ a smooth non-negative function obeying
\begin{align}\label{def:phi}
\phi(\lambda)=1 \quad \text{for}\quad  0\leq \lambda \leq 1 \quad \text{and} \quad \phi(\lambda)=0 \quad  \text{for}\quad   \lambda \geq 2.
\end{align}
For each dyadic number $N \in 2^{\Z}$, we define 
\begin{align}\label{def:phi scal}
\phi_N(\lambda): = \phi(\lambda/N)\quad \text{and} \quad \psi_N(\lambda): = \phi_N(\lambda) - \phi_{N/2}(\lambda).
\end{align}
Notice that $\big\{ \psi_N(\lambda)\big\}_{N\in 2^{\Z}}$ forms a partition of unity for $(0, \infty)$. With these functions, we define the Littlewood-Paley projections:
\begin{align*}
P^{\Omega}_{\leq N}f : = \phi_N(\sqrt{-\Delta_{\Omega}})f , \quad  P^{\Omega}_{N}f:= \psi_N(\sqrt{-\Delta_{\Omega}})f, \quad  P^{\Omega}_{>N}f:= I - P^{\Omega}_{\leq N} f, 
\end{align*}
and
\begin{align*}
\tilde{P}^{\Omega}_{N}f:= &\; \tilde{\psi}_N\left(\sqrt{-\Delta_{\Omega}}\right)f   \\
= &\;  \psi_{N-1}\left(\sqrt{-\Delta_{\Omega}}\right)f + \psi_{N}\left(\sqrt{-\Delta_{\Omega}}\right)f +\psi_{N+1}\left(\sqrt{-\Delta_{\Omega}}\right)f. 
\end{align*}

We introduce the homogeneous Besov space as the following.
\begin{definition}\label{def:Besov}Let $s\in \R$ and $1\leq q, r \leq \infty$. The homogeneous Besov space $\dot B^{s}_{D, q, r}(\Omega)$ consists of the distributions $f$ supported on $\Omega$ such that
	\begin{align*}
	\|f\|_{\dot B^{s}_{D, q, r}(\Omega)} : = \left(\sum_{N\in 2^{\Z}} N^{sr} \|P^{\Omega}_{N} f\|^{r}_{L^{q}(\Omega)}\right)^{1/r}<\infty.
	\end{align*}
	\end{definition}

\begin{lemma}[\cite{KVZ16:IMRN Riesz tf, LiSZ:NLS}]\label{lem:berns ests}
	For any radial function $f\in C^{\infty}_{c}(\Omega)$, we have
	\begin{align}
	\left\|P^{\Omega}_{\leq N}f \right\|_{L^{p}(\Omega)}  +  &	\left\|P^{\Omega}_{ N}f \right\|_{L^{p}(\Omega)}  \lesssim  	\left\| f \right\|_{L^{p}(\Omega)}, \label{est:LP bd1} \\
	N^s	\left\|P^{\Omega}_{N} f \right\|_{L^{p}(\Omega)}  & \approx  \left\|\big(-\Delta_{\Omega}\big)^{s/2}P^{\Omega}_{N} f \right\|_{L^{p}(\Omega)} \label{est:LP bd2}
	\end{align}
	for any $1\leq p \leq \infty$ and $s\in \R$, Moreover, we have
		\begin{align*}
	\left\|P^{\Omega}_{\leq N} f\right\|_{L^{q}(\Omega)}  + 
	\left\|P^{\Omega}_{N} f\right\|_{L^{q}(\Omega)} \lesssim N^{3\big(\frac{1}{p} -\frac{1}{q}\big)} 	\left\| f \right\|_{L^{p}(\Omega)} 
	\end{align*}
	for all $1\leq p\leq q\leq \infty$. The implicit constants depend only on $p, q$ and $s$.
	\end{lemma}
In general case, the estimates \eqref{est:LP bd1} and \eqref{est:LP bd2} only hold for $1<p<\infty$ by Theorem \ref{thm:multiplier}. In the radial case, the corresponding integral kernels have good properties by the distorted Fourier transform, see more details in \cite{LiSZ:NLS}.  
Therefore, for any $1<p<\infty$ and any radial $f\in L^p(\Omega)$, we have  the following homogeneous decomposition
	\begin{align*}
	f(x) = \sum_{N\in 2^{\Z}}  P^{\Omega}_{N} f(x).
 	\end{align*}
 In particular, the sums converge in $L^p(\Omega)$.
 
 \begin{definition}\label{def:Sob DB}
 	For $s\geq 0$ and $1<p<\infty$, Let $\dot H^{s,p}_{D}(\Omega) $  and $ H^{s,p}_{D}(\Omega) $  denote the completions of $C^{\infty}_{c}(\Omega)$ under the norms 
 	\begin{align*}
 	\|f\|_{\dot H^{s,p}_{D} }: = \left\| \big( -\Delta_{\Omega}\big)^{s/2} f \right\|_{L^{p}(\Omega)} \quad \text{and}\quad
 	\|f\|_{ H^{s,p}_{D} }: = \left\| \big( I-\Delta_{\Omega}\big)^{s/2} f \right\|_{L^{p}(\Omega)}. 
 	\end{align*}
 	When $p=2$, we write $\dot H^{s}_{D}(\Omega) $ and $ H^{s}_{D}(\Omega) $  for $\dot H^{s,2}_{D}(\Omega) $ and $ H^{s,2}_{D}(\Omega) $, respectively. 
 \end{definition}
 We will use the radial Sobolev spaces $\dot H^{s,p}_{D, rad}(\Omega) $  and $ H^{s,p}_{D, rad}(\Omega) $ in the context.  If the radial function $f$ belongs to $ C^{\infty}_{c}(\Omega)$, then 
 \begin{align*}
 \mathcal{F}_D  (-\Delta_{\Omega} f ) (\lambda) =  \lambda^2 \mathcal{F}_D  ( f ) (\lambda), 
 \end{align*}
and  $\mathcal{F}_D $ induces an isometric map 
 \begin{align*}
 \mathcal{F}_D : \; \dot H^1_{D, rad}(\Omega) \longrightarrow L^{2}\left([0, \infty), \lambda^2 \, d\lambda\right).
 \end{align*}
 
 The Littlewood-Paley square function estimates and the dense result then follow from the multiplier theorem by the usual argument in \cite{Gr14:CFA, KVZ16:IMRN Riesz tf, Stein:book:SI}.
 \begin{proposition}[\cite{KVZ16:IMRN Riesz tf}]\label{prop:squ funt} 
 	Fix $1<p<\infty$ and $s\geq 0$.  Then for any $f\in C^{\infty}_{c}(\Omega)$,  we have
 	\begin{align*}
 	\left\| \big(-\Delta_{\Omega}\big)^{s/2} f \right\|_{L^p(\Omega)} \approx	\left\| \left(  \sum_{N\in 2^{\Z}} N^{2s} \big|  P^{\Omega}_{N} f  (x)\big|^2\right)^{1/2}  \right\|_{L^p(\Omega)}.  
 	\end{align*}
 \end{proposition}
 
 \begin{proof}For the convenience to the reader, we present the proof here. It follows the argument in \cite{KVZ16:IMRN Riesz tf}.  
 It suffices to show that for all $f\in L^p(\Omega)$, the following estimate holds
  	\begin{align*}
  	\left\| S(g)\right\|_{L^p(\Omega)} \approx  \left\| g\right\|_{L^p(\Omega)} , \;\; \text{where}\;\; S(g) = 
 \left(  \sum_{N\in 2^{\Z}} N^{2s} \big|  P^{\Omega}_{N} \left(-\Delta_{\Omega}\right)^{-s/2}g. \big|^2\right)^{1/2}  
 \end{align*}
In fact, one can apply the above equivalent relation to $g= \left(-\Delta_{\Omega}\right)^{s/2}f$ with $f\in C^{\infty}_{c}(\Omega)$. 

We first show that $	\left\| S(g)\right\|_{L^p(\Omega)} \lesssim  \left\| g\right\|_{L^p(\Omega)}$. Note that 
\begin{align*}
N^{s}  P^{\Omega}_{N} \left(-\Delta_{\Omega}\right)^{-s/2} = m \left(\frac{1}{N}\sqrt{-\Delta_{\Omega}}\right), \;\; \text{with}\;\; m(\lambda):=\lambda^{-s} \psi_1 (\lambda)
\end{align*}
where $\psi_1$ is defined in \eqref{def:phi scal}, and for all integers $k\geq 0$, we have $\left| \lambda^k \partial^k_{\lambda} m(\lambda) \right| \lesssim 1$. 
Therefore, the multiplier
\begin{align*}
m_{\epsilon}(\lambda): = \sum_{N\in 2^{\Z} }\epsilon_N m\left(\frac{\lambda}{N}\right)
\end{align*}
satisfies that 
$
\left| \lambda^k \partial^k_{\lambda} m_{\epsilon}(\lambda) \right| \lesssim 1
$
uniformly in the choice of signs $\{\epsilon_N\} \subset \{\pm 1\}$.  (only finitely many terms of the summands give nonzero contribution due to the compact support of the function $\psi_1$). 

Applying the Khintchine inequality, Fubini and Theorem \ref{thm:multiplier}, we obtain that 
\begin{align*}
\int_{\Omega} \left| S(g) (x)\right|^p \; dx  \lesssim \int_{\Omega} \mathbb {E} \left\{  \left|(m_{\epsilon}g)(x)  \right|^p \right \} \; dx =  \mathbb {E}  \left\|(m_{\epsilon}g)(x) \right\|^{p}_{L^p(\Omega)} \lesssim \| g\|^P_{L^p(\Omega)}. 
\end{align*}
This gives that $	\left\| S(g)\right\|_{L^p(\Omega)} \lesssim  \left\| g\right\|_{L^p(\Omega)}$. 

Next, we show the reverse inequality $ \left\| g\right\|_{L^p(\Omega)} \lesssim 	\left\| S(g)\right\|_{L^p(\Omega)}$ by the duality argument. It is obvious that  the multiplier 
\begin{align*}
\tilde{m}(\lambda):= \left( \sum_{N\in 2^{\Z} } \left[ m\left(\frac{\lambda}{N}\right)\right]^2 \right)^{-1}
\end{align*}
satisfies the assumption of Theorem \ref{thm:multiplier}, hence it defines another bounded multiplier. By the Cauchy-Schwarz inequality, we have 
\begin{align*}
\big|\left< g, h\right> \big| = & \;  \left|\sum_{N\in 2^{\Z} } \left< g, \left[ m\left(\frac{1}{N}\sqrt{-\Delta_{\Omega}}\right)\right]^2\tilde{m}\left(\sqrt{-\Delta_{\Omega}}\right) h\right>  \right|\\
= & \;  \left|\sum_{N\in 2^{\Z} } \left< m\left(\frac{1}{N}\sqrt{-\Delta_{\Omega}}\right) g,  m\left(\frac{1}{N}\sqrt{-\Delta_{\Omega}}\right)\tilde{m}\left(\sqrt{-\Delta_{\Omega}}\right) h\right>  \right|\\
\leq  & \;   \left<S(g),  S\left(\tilde{m}\left(\sqrt{-\Delta_{\Omega}}\right) h\right)\right>  \\
\leq  & \;   \left\|S(g)\right\|_{L^{p}(\Omega)}  \left\|S\left(\tilde{m}\left(\sqrt{-\Delta_{\Omega}}\right) h\right)\right\|_{L^{p'}(\Omega)} \\
\lesssim  & \;   \left\|S(g)\right\|_{L^{p}(\Omega)}  \left\|\tilde{m}\left(\sqrt{-\Delta_{\Omega}}\right) h\right\|_{L^{p'}(\Omega)} \\
\lesssim  &  \;   \left\|S(g)\right\|_{L^{p}(\Omega)}  \left\|  h\right\|_{L^{p'}(\Omega)},
\end{align*}
which shows the reverse inequality $ \left\| g\right\|_{L^p(\Omega)} \lesssim 	\left\| S(g)\right\|_{L^p(\Omega)}$  by the duality and completes the proof. 
 	\end{proof}

 \begin{proposition}[\cite{KVZ16:IMRN Riesz tf}]\label{prop:dense}
 	For $1<p<\infty$, and $s<1+\frac{1}{p}$, $\dot H^{s,p}_{D}(\Omega)$ is dense in $L^p(\Omega)$.
 \end{proposition}

\begin{remark}
	The condition $s<1+\frac{1}{p}$ is essential, and the above result is used to show the boundedness of Riesz transform on the exterior domain together with the boundedness of classical Riesz transform \cite{Gr14:CFA, MuscaluS:book1, Stein:book:SI} and Theorem \ref{thm:Sob eq}. We can see the details in the proof of Lemma $4.4$ in \cite{KVZ16:IMRN Riesz tf}, and reference therein.  
	\end{remark}
 
 \begin{definition}\label{def:Sob 0 extd}
 	The space  $\dot H^{s,p}_{00}(\Omega) $ and  $ H^{s,p}_{00}(\Omega) $ are defined as the completion of $C^{\infty}_{c}(\Omega)$ in $\dot H^{s,p}(\R^3)$ and  $ H^{s,p}(\R^3)$, respectively.
 \end{definition}

\begin{proposition}[\cite{KVZ16:IMRN Riesz tf}]\label{prop:Hardy for D-Lap}
	Let $1<p<\infty$ and $0<s<\min(1+\frac{1}{p}, \frac{3}{p})$. Then for any $f\in C^{\infty}_{c}(\Omega)$, we have
	\begin{align*}
	\left\| \frac{f(x)}{\text{dist}(x, \Omega^c)}\right\|_{L^{p}(\Omega)}\lesssim \left\| \left(-\Delta_{\Omega}\right)^{s/2}f \right\|_{L^{p}(\Omega)}. 
	\end{align*}
\end{proposition}
\begin{remark} The condition $s< \frac{3}{p}$ is essential for the Hardy inequalities in both $\R^3$ and  the exterior domain $\Omega$, we can refer to \cite{BahChe:book,KVZ16:IMRN Riesz tf}.
	
\end{remark}

At last, the equivalence between $\dot H^{s,p}_{D}(\Omega) $ and $\dot H^{s,p}_{00}(\Omega)$ with proper exponents then follows from Hardy's inequalities in the exterior domain $\Omega$ and the whole space $\R^3$ in \cite{KVZ16:IMRN Riesz tf}.

\begin{theorem}[\cite{KVZ16:IMRN Riesz tf}]\label{thm:Sob eq}
	Suppose $1<p<\infty$ and $0\leq s < \min\{1+\frac1p, \frac{3}p\}$, then for all $f\in C^{\infty}_{c}(\Omega)$
	\begin{align*}
	\left\| \big( -\Delta_{\Omega}\big)^{s/2} f \right\|_{L^{p}(\Omega)} \approx_{p,s}  	 \left\| \big( -\Delta_{\R^3}\big)^{s/2} \tilde{f} \right\|_{L^{p}(\R^3)}
	\end{align*}
where  $\tilde{f}=f $ on $\Omega$, and $\tilde{f}=0$ outside $\Omega$. 	Thus $\dot H^{s,p}_{D}(\Omega) = \dot H^{s, p}_{00}(\Omega)$ for these values of the parameters.
\end{theorem}

\begin{remark}\label{rem:Frac prod}
 On the one hand, the condition $s<\min(1+\frac{1}{p}, \frac{3}{p})$ is  necessary in Theorem \ref{thm:Sob eq} since it is essential in Proposition \ref{prop:dense} and Proposition \ref{prop:Hardy for D-Lap}.  We can see the counterexample of the corresponding Riesz transform outside a convex obstacle for the case $s\geq \min(1+\frac{1}{p}, \frac{3}{p})$ in \cite{KVZ16:IMRN Riesz tf}. 
	
	On the other hand, as the direct corollary of Theorem \ref{thm:Sob eq}, the fractional product rule directly follows from the classical Euclidean setting. More precisely, for all $f, g\in C^{\infty}_{c}(\Omega)$, then
	\begin{align*}
		\left\| \big( -\Delta_{\Omega}\big)^{s/2} \big( f g\big) \right\|_{L^{p}(\Omega)} \lesssim 		\left\| \big( -\Delta_{\Omega}\big)^{s/2} f \right\|_{L^{p_1}(\Omega)} \|g\|_{L^{q_1}(\Omega)}+ 	\|f\|_{L^{q_2}(\Omega)}\left\| \big( -\Delta_{\Omega}\big)^{s/2} g \right\|_{L^{p_2}(\Omega)}
	\end{align*}
	with the exponents satisfying $1<p, p_1, p_2<\infty$, $1<q_1, q_2 \leq \infty$, 
	\begin{align*}
	\frac{1}{p} = \frac{1}{p_1} + \frac{1}{q_1}= \frac{1}{p_2} + \frac{1}{q_2}, \quad 0<s<\min\left(1+\frac{1}{p_1}, 1+\frac{1}{p_2}, \frac{3}{p_1}, \frac{3}{p_2}\right). 
	\end{align*}
And the fractional chain rule holds in a similar way. 
\end{remark}

\section{Linear estimates}\label{sect:linear est}
In this part, we will make use of the distorted Fourier transform adapted to the Dirichlet-Laplacian operator $-\Delta_{\Omega}$ to show the dispersive estimate and the global-in-time (endpoint) Strichartz estimate of the linear wave equation with zero Dirichlet boundary value outside of the unit ball with radial data,  we can also refer to \cite{LiSZ:NLS, LiXZ:NLS} for the dispersive and Strichartz estimates of the Schr\"odinger equation outside of the unit ball with radial data. 

Now we consider the 3D Dirichlet boundary value problem of the linear wave equation with radial data
\begin{equation}\label{eq:line wave}
\begin{cases}
\partial^2_t u -\Delta u = F, & (t,x)\in\R \times\Omega, \\
u(0,x)=u_0(x),  & x\in \Omega, \\
\partial_t u(0, x)= u_1(x), & x\in \Omega, \\
u(t,x)=0, & x\in \partial \Omega, 
\end{cases}
\end{equation}
where $\Omega =\R^3\backslash \overline{B(0, 1)}$ and initial data $u_0, u_1$ and the inhomogeneous term $F$ are radial in $x$.  By the functional calculus, we have
\begin{equation}\label{lin wave}
u(t,x)=\cos(t\sqrt{-\Delta_{\Omega}})u_0+\frac{\sin(t\sqrt{-\Delta_{\Omega}})}{\sqrt{-\Delta_{\Omega}}} u_1 + \int^{t}_{0} \frac{\sin((t-s)\sqrt{-\Delta_{\Omega}})}{\sqrt{-\Delta_{\Omega}}} F(s, x)\, ds.
\end{equation}

Let us denote the half-wave operator as 
\begin{equation}\label{def:half wave}
U(t) = e^{it\sqrt{-\Delta_{\Omega}}},
\end{equation}
then 
\begin{align}\label{relat:wave and half wave}
\cos(t\sqrt{-\Delta_{\Omega}})u_0 = 
\frac{U(t)+U(-t)}{2} u_0, \quad 
\frac{\sin(t\sqrt{-\Delta_{\Omega}})}{\sqrt{-\Delta_{\Omega}}} u_1 = \frac{U(t)-U(-t)}{2i\sqrt{-\Delta_{\Omega}}} u_1.
\end{align}

By the distorted Fourier transform and  the stationary phase estimate, we have the following uniform dispersive estimate in the radial case.
\begin{proposition}\label{prop:Disp est}
	Let $2 \leq r \leq \infty$,  and the  radial function $f$ is supported on $\Omega$, then
	\begin{align*}
	\big\| U(t)  f \big\|_{\dot B^{-\beta(r)}_{D, r,2} (\Omega)} \lesssim |t|^{-\gamma(r)}  	\big\| f \big\|_{\dot B^{\beta(r)}_{D, r',2} (\Omega)} 
	\end{align*}
	where $\beta(r)=\gamma(r)=1- \frac2r$.
\end{proposition}
\begin{remark}
In higher dimensions $d\geq 4$, the eigenfunction no longer have the simple form \eqref{quan:eign funt}, which will induce more complexity. We can refer to \cite{LiXZ:NLS} for  the dispersive estimates for the Schr\"odinger equation in dimensions $n=5, 7$,  and more details. 
\end{remark}
\begin{proof}Taking the distorted Fourier transform, we have
	\begin{align}\label{id:cons law}
	\big\| U(t) f \big\|_{L^2(\Omega) } = 	\big\|f \big\|_{L^2(\Omega) }. 
	\end{align}
Therefore, by the interpolation theorem, it suffices to show the  following uniform dispersive estimate
		\begin{align}\label{est:dis est:L1}
	\big\| U(t)  f \big\|_{\dot B^{-1}_{D, \infty,2} (\Omega)} \lesssim |t|^{-1}  	\big\| f \big\|_{\dot B^{1}_{D, 1,2} (\Omega)}. 
	\end{align}
	
Notice  that $	P^{\Omega}_{N}f (r)  = 	\tilde{P}^{\Omega}_{N} P^{\Omega}_{N} f (r) $ for $N\in 2^{\Z}$, $r=|x|\geq 1$, we have
	\begin{align}
	U(t) P^{\Omega}_{N}f (r) = & \frac{2}{\pi} \int^{\infty}_{0} \int^{\infty}_{1} e_{\lambda}(r)  e_{\lambda}(s) e^{it |\lambda|} \tilde{\psi}_{N}(\lambda) \big(P^{\Omega}_{N}f\big)(s) s^2  \; ds\, d\lambda \nonumber\\
	= &  \int^{\infty}_{1} K_N(t,r; s) \big(P^{\Omega}_{N}f\big)(s)  s^2 \; ds,  \label{id:Ut op kernel}
	\end{align}
where the integral kernel $K_N(t,r; s)$ is 
		\begin{align*}
	K_N(t,r; s) = & \frac{2}{\pi} \int^{+\infty}_{0}   e_{\lambda}(r)  e_{\lambda}(s) e^{i \lambda t} \tilde{\psi}_{N}(\lambda) \, d\lambda\\
	= &   \frac{2}{\pi} \int^{+\infty}_{0}  \frac{\sin \lambda(r-1)}{r}\cdot \frac{\sin \lambda(s-1)}{s}\cdot  e^{i\lambda t} \cdot \tilde{\psi}_N(\lambda)\; d\lambda\\
	= & \frac{2N^3}{\pi}\cdot \frac{s-1}{s}\cdot \frac{r-1}{r} \cdot \int_{0}^{+\infty}  \frac{\sin N\lambda(s-1)}{N(s-1)}\cdot \frac{\sin N\lambda(r-1)}{N(r-1)} \cdot e^{iN \lambda t}\cdot \tilde{\psi}_1(\lambda)d\lambda\\
	\approx & \;  N^3\cdot \frac{s-1}{s} \cdot \frac{r-1}{r}\cdot \left(e^{iNt{\sqrt{-\Delta_{rad}}}}\tilde{P}_1\right)(N(r-1); N(s-1))
\end{align*}
and  the kernel $\left(e^{it{\sqrt{-\Delta_{rad}}}}\tilde{P}_1\right)(r;s)$   is defined by
\begin{align*}
\left(e^{it{\sqrt{-\Delta_{rad}}}}\tilde{P}_1\right)(r;s) =\; \text{Const}\cdot \int_{0}^{+\infty}\frac{\sin\lambda s}{s}\cdot \frac{\sin\lambda r}{r}\cdot e^{it\lambda}\cdot  \tilde{\psi}_1(\lambda)\,d\lambda, \quad r,\, s>0,
\end{align*}
which is related to  the usual radial half-wave propagator in the whole space $ \R^3$, and  follows from the properties of the Bessel function (see the proof  in Appendix \ref{app:A}). 

 By the stationary phase estimate in \cite{Gr14:CFA, MuscaluS:book1} (we can also refer to \cite{GinV:NLW:St est, KeelTao:Endpoint, Tao:book}), we obtain
	\begin{align}\label{est:Kernel}
	\sup_{r, s\geq 1} \big| K_N(t, r; s)\big|\lesssim   N^{3}\, |Nt|^{-1} 
	\lesssim  |t|^{-1} N^{2}. 
	\end{align}
By  \eqref{id:Ut op kernel}, \eqref{est:Kernel} and the Minkowski inequality, we obtain	
	\begin{align*}
	\big\| U(t)  P^{\Omega}_{N} f \big\|_{L^{\infty}(\Omega)}\lesssim N^3 \cdot  (N|t|)^{-1}  \cdot \big\| P^{\Omega}_{N} f \big\|_{L^1 (\Omega)} \lesssim |t|^{-1}  N^{2}	\big\| P^{\Omega}_{N} f \big\|_{L^1 (\Omega)}, 
	\end{align*}
	which implies the result and completes the proof.
	\end{proof}

By the $TT^{*}$ dual argument in \cite{GinV:NLW:St est},  and  the
 Christ-Kiselev Lemma in \cite{ChrstK:lem},   the dispersive estimate of the operator $e^{it\sqrt{-\Delta_{\Omega}}} $ in Proposition \ref{prop:Disp est} together with the conservation law \eqref{id:cons law} implies the Strichartz estimate of \eqref{eq:line wave}. We can also refer to \cite{KeelTao:Endpoint, Tao:book}. 
\begin{theorem}[Strichartz estimate with radial data]\label{thm:St est}
	Let $\rho_1, \rho_2, \mu\in \R$ and $2\leq q_1, q_2, r_1, r_2\leq \infty$ and let the following conditions be satisfied
	\begin{align*}
	0\,\leq \, \frac{1}{q_i} +\frac{1}{r_i} \leq \,\frac12, & \quad r_i \not = \infty,   \quad  i=1,\, 2, \\
	\rho_1 +3\left(\frac12 -\frac{1}{r_1}\right)-\frac{1}{q_1} = \mu, & \quad 
	\rho_2  +3\left(\frac12 -\frac{1}{r_2}\right)-\frac{1}{q_2} = 1-\mu. 
	\end{align*} 
Let $u_0, u_1$ and $F$ be radial in $x$, and $u: I\times \Omega\longrightarrow \R$ be  a solution to linear wave equation \eqref{eq:line wave} with $0\in I$. Then $u$ satisfies the estimates
	\begin{align*}
	\left\|u\right\|_{L^{q_1}_I\dot B^{\rho_1}_{D, r_1, 2}(\Omega)\cap C(I; \dot H^{\mu}_D(\Omega))} + 	\left\|\partial_t u\right\|&_{L^{q_1}_I\dot B^{\rho_1-1}_{D, r_1, 2}(\Omega) \cap C(I; \dot H^{\mu-1}_D(\Omega)) }   \\
	& \lesssim \; \|(u_0, u_1)\|_{\dot H^{\mu}_{D, rad}(\Omega)\times \dot H^{\mu-1}_{D,rad}(\Omega)} + \|F\|_{L^{q'_2}_I\dot B^{-\rho_2}_{D, r'_2, 2}(\Omega)}.
	\end{align*}
	\end{theorem}

\begin{remark}
By the explicite distorted Fourier transform,  we give a short proof of the dispersive estimate and the global-in-time Strichartz estimate of the 3D Dirichlet boundary value problem of linear wave equation outside of the unit ball in the radial case. In general case, the related estimates are extremely complicated, please refer to \cite{Burq03: NLW CPDE, DuyYang24:LW Str,  HidanoMSSZ:AST est, Iv:NLS St est, SmithSogge00:NLW CPDE, SmithSW:AST est} and  \cite{IvLasLebP:LW:dis est, IvLebeauP:LW:dis est} for more details. 
	\end{remark}

\begin{proof}[Proof of Theorem \ref{thm:St est}]For the convenience to the reader,
we combine the  argument as that in the proof of Proposition $3.1$ in \cite{GinV:NLW:St est} and  the Christ-Kiselev Lemma in \cite{ChrstK:lem} to sketch the proof here.	By use of \eqref{lin wave}, \eqref{def:half wave} and \eqref{relat:wave and half wave}, it suffices to show that the half-wave operator satisfies 
\begin{align}\label{lin est:hom}
\left\| U(t)f(x)\right\|_{L^{q_1}_{\R}\dot B^{\rho_1}_{D, r_1, 2}(\Omega)} \leq \;  C \|f\|_{L^2\left(\Omega\right)}, 
\end{align}
and
\begin{align}
\left\| \int_{\R}U(t-s)F(s,x)\, ds \right\|_{L^{q_1}_{\R}\dot B^{\rho_1}_{D, r_1, 2}(\Omega)} \leq & \;  C \|F\|_{L^{q'_2}_{\R}\dot B^{-\rho_2}_{D, r'_2, 2}(\Omega)}, \label{lin est:dual}\\
\left\| \int^{t}_{0}U(t-s)F(s,x)\, ds \right\|_{L^{q_1}_{I}\dot B^{\rho_1}_{D, r_1, 2}(\Omega)} \leq & \;  C \|F\|_{L^{q'_2}_{I}\dot B^{-\rho_2}_{D, r'_2, 2}(\Omega)},  \label{lin est:inhom}
\end{align}
where $I=[0, T]\subset [0,+\infty)$ and  the functions $f$ and $F$ are radial in $x$,  under the conditions 
	\begin{align*}
0\,\leq \, \frac{1}{q_i} +\frac{1}{r_i} \leq \,\frac12, & \quad r_i \not = \infty,   \quad  i=1,\, 2,  \\
\rho_1 +3\left(\frac12 -\frac{1}{r_1}\right)-\frac{1}{q_1} = 0, & \quad 
\rho_2  +3\left(\frac12 -\frac{1}{r_2}\right)-\frac{1}{q_2} = 0.   
\end{align*} 

By the interpolation, it is equivalent to show  
\begin{align}
\left\| \int_{\R}U(t-s)F(s,x)\, ds \right\|_{L^{q_1}_{\R}\dot B^{\rho_1}_{D, r_1, 2}(\Omega)} \leq & \;  C \|F\|_{L^{q'_1}_{\R}\dot B^{-\rho_1}_{D, r'_1, 2}(\Omega)}, \label{lin est:dual2}\\
\left\| \int^{t}_{0}U(t-s)F(s,x)\, ds \right\|_{L^{q_1}_{I}\dot B^{\rho_1}_{D, r_1, 2}(\Omega)} \leq & \;  C \|F\|_{L^{q'_1}_{I}\dot B^{-\rho_1}_{D, r'_1, 2}(\Omega)},  \label{lin est:inhom2}
\end{align}
to prove \eqref{lin est:dual} and \eqref{lin est:inhom}. 

On the one hand, the estimate \eqref{lin est:hom} is equivalent to the estimate   \eqref{lin est:dual2} by the $ TT^{*}$ dual argument in \cite{GinV:NLW:St est}, and on the other hand, the retarded estimate \eqref{lin est:inhom2} can be deduced from the estimate \eqref{lin est:dual2} by the Christ-Kiselev Lemma in \cite{ChrstK:lem}. 	

At last, by  the Hardy-Littlewood-Sobolev inequality,  we can obtain the estimate \eqref{lin est:dual2}  from the uniform dispersive estimate in Proposition \ref{prop:Disp est}. This completes the proof. 
\end{proof}

 In fact, we can also follow the argument in \cite{Stb:LW St} and obtain improved endpoint Strichartz estimates in the radial case, which depends on the non-uniform dispersive estimate. As for the endpoint $L^2_tL^{\infty}_x$ estimate for the radial case in $\R^3$, we can also refer to \cite{KlainM:endp ST est, Tao:endpoint ST est}. 

\begin{theorem}[Endpoint Strichartz estimates with radial data]\label{thm:endp est} Let $u_0, u_1$ and $F$ be radial in $x$ variable, and $u: I\times \Omega\longrightarrow \R$ be  a solution to linear wave equation \eqref{eq:line wave} with $0\in I$. If $q>4$ and 
$
 s  = 1-3/q, 
$
	then we obtain
	\begin{align*} 
	\left\|u\right\|_{L^2_tL^q_x(I\times \Omega)}  \lesssim \|(u_0,u_1)\|_{\dot H^s_{D, rad}(\Omega)\times \dot H^{s-1}_{D, rad}(\Omega)} + \|F\|_{L^1_t \left(I;\,\dot H^{s-1}_{D, rad}(\Omega)\right)}. 
	\end{align*}
	
\end{theorem}

\begin{proof}By the energy estimate, it suffices to show the following homogeneous estimate for the half wave operator $U(t)$
	\begin{align*}
	\big\| U(t)f\big\|_{L^2_tL^q_x(I \times \Omega)} \lesssim \|f\|_{\dot H^s_{D, rad}(\Omega)}. 
	\end{align*}

By the distorted Fourier transform once again, we have
	\begin{align*}
U(t)f(r) = & \; \frac{2}{\pi}\int^{\infty}_{0} \frac{\sin(\lambda(r-1))}{r}\, e^{it|\lambda|}\left( \mathcal{F}_{D} f\right) (\lambda)\, d\lambda \\
	= &\;  \frac{1}{i\pi }\int^{\infty}_{0} \frac{e^{i\lambda(r-1)}-e^{-i\lambda(r-1)}}{r}\, e^{it|\lambda|}\left( \mathcal{F}_{D} f\right) (\lambda)\, d\lambda.
	\end{align*}
By  the Littlewood-Paley decomposition, we have
	\begin{align*}
		U(t)f(r)= &\sum\limits_{N} e^{it\sqrt{-\Delta_{\Omega}}}\tilde{P}^{\Omega}_{N}P^{\Omega}_{N}f(r) \\ 
	= & \sum\limits_{N} \frac{1}{i\pi r}\int^{\infty}_{0} \left( e^{i\lambda(r-1)}-e^{-i\lambda(r-1)}\right)\, e^{it|\lambda|} \tilde{\psi}_N(\lambda)\left( \mathcal{F}_{D} P^{\Omega}_{N}f\right) (\lambda)\, d\lambda\\
	= & \sum\limits_{N} \frac{N}{i\pi r}\int^{\infty}_{0} \left( e^{iN\lambda(r-1)}-e^{-iN\lambda(r-1)}\right)\, e^{itN|\lambda|} \tilde{\psi}_0(\lambda)\left( \mathcal{F}_{D} P^{\Omega}_{N}f\right) (N\lambda)\, d\lambda.
\end{align*}
Notice that  $\mathcal{F}_{D} (P^{\Omega}_{N}f)$ is supported on $(N/2, 2N)$, we can make a decomposition by Fourier series as the following,
	\begin{align*}
	 \left( \mathcal{F}_{D} P^{\Omega}_{N}f\right) (N\lambda)=\sum_{k\in \Z} c_k^N e^{ik\frac{2\pi}{2N}N\lambda},
	\end{align*}
then we have
			\begin{align*}
	e^{it\sqrt{-\Delta_{\Omega}}}\tilde{P}^{\Omega}_{N}P^{\Omega}_{N}f(r)  
	= &  \sum_{k\in \Z} \frac{\, Nc_k^N}{i\pi r} \int^{\infty}_{0} \left( e^{iN\lambda(r-1)}-e^{-iN\lambda(r-1)}\right)\, e^{itN|\lambda|} \tilde{\psi}_0(\lambda)  e^{i\pi k\lambda}\, d\lambda\\
	= &   \sum_{k\in \Z} \frac{\, Nc_k^N}{i\pi r} \Big(	\psi^{+}_{k}(t,r) -	\psi^{-}_{k}(t,r)\Big) ,
	\end{align*}
where the functions $\psi^{\pm}_{k}(t,r)$  are defined by
	\begin{align*}
	\psi^{\pm}_{k}(t,r) := \int^{\infty}_{0}  e^{\pm iN\lambda(r-1)} \, e^{itN \lambda} \tilde{\psi}_0(\lambda)  e^{i \pi k\lambda}\, d\lambda.
	\end{align*}
	
	On the one hand, by the support property of the cut-off function $\tilde{\psi}_0$,  we can obtain the boundness of the integral 
	\begin{equation*}
	\left|	\psi^{\pm}_{k}(t,r) \right| = \left| \int^{\infty}_{0}  e^{\pm iN\lambda(r-1)} \, e^{itN \lambda} \tilde{\psi}_0(\lambda)  e^{i \pi k\lambda}\, d\lambda \right| \lesssim C. 
	\end{equation*}
On the other hand, for any $M\geqslant 1$,  by integrating by parts $M$ times, we have 
	\begin{equation*}
	\left|	\psi^{\pm}_{k}(t,r) \right| \lesssim \frac{C_M}{ | Nt+\pi k\pm N (r-1)|^{M}}.
	\end{equation*}
Therefore,   for any $M\geqslant 1$,  we have the local dispersive estimate
\begin{equation*}
\left|	\psi^{\pm}_{k}(t,r) \right| \lesssim \frac{C_M}{\big( 1+ N|t+\frac{\pi k}{N}\pm (r-1)|\big)^{M}}.
\end{equation*}

Then  we can make the simple calculation and obtain
    \begin{align*}
	\left\|	e^{it\sqrt{-\Delta_{D}}}P^{\Omega}_{N}f \right\|^{q}_{L^q_x(\Omega)} \lesssim &  \int^{\infty}_{1} \left|  \sum_{k\in \Z} c_k^N\,  \psi^{\pm}_{k}(t,r) \right|^q\frac{1}{r^q} r^2 \; dr \\
	\lesssim & \int^{\infty}_{1} \left|  \sum_{k\in \Z} \frac{|c_k^N|}{\big( 1+ N|t+\frac{\pi k}{N}\pm (r-1)|\big)^{M}}  \right|^q  r^{2-q} \; dr \\
	\lesssim & \sum_{k\in \Z}  \int^{\infty}_{0}   \frac{|c_k^N|^q}{\big( 1+ N||t+\frac{\pi k}{N}|-r|\big)^{2}}    (1+r)^{2-q} \; dr,
\end{align*}
where we choose $M$ such that  $(M-2/q)\cdot q' >1$ and  use the Young inequality in the last inequality.	 We have for $s=1-\frac{3}{q}$ that
\begin{align*}
	\left\|	e^{it\sqrt{-\Delta_{D}}}P^{\Omega}_{N}f \right\|_{L^q_x(\Omega)}\lesssim & \left(\sum_{k\in \Z}  \int^{\infty}_{0}   \frac{|c_k^N|^q}{\big( 1+ N||t+\frac{\pi k}{N}|-r|\big)^{2}}    (1+r)^{2-q} \; dr\right)^{\frac{1}{q}}\\
	\lesssim & \left(\sum_{k\in \Z}  N^{q-3}\int^{\infty}_{0}   \frac{|c_k^N|^q}{\big( 1+ ||Nt+\pi k|-r|\big)^{2}} (1+r)^{2-q} \; dr\right)^{\frac{1}{q}}\\
	\lesssim & N^{s} \left(\sum_{k\in \Z} \frac{|c_k^N|^q}{\big(1+|Nt+\pi k|\big)^{q-2}}\right)^\frac{1}{q}\\
	\lesssim &  N^{s} \left(\sum_{k\in \Z} \frac{|c_k^N|^2}{\big(1+|Nt+\pi k|\big)^{2-\frac{4}{q}}}\right)^\frac{1}{2},
\end{align*}
where we use the embedding fact that $l^2\subset l^q$ in the last inequality. By Bernstein's inequality, we have for $q>4$ that
\begin{align*}
	\left\|	e^{it\sqrt{-\Delta_{D}}}P^{\Omega}_{N}f \right\|^2_{L^2_tL^q_x(I\times \Omega)}\lesssim  N^{2s}\int_I \sum_{k\in \Z} \frac{|c_k^N|^2}{\big(1+|Nt+\pi k|\big)^{2-\frac{4}{q}}}\,dt
	\lesssim  N^{2s}\sum_{k\in \Z} |c_k^N|^2
	\lesssim  \|P_N^\Omega f\|^2_{\dot{H}^s_{D}(\Omega)},
\end{align*}
which together with Proposition \ref{prop:squ funt} implies the result and completes the proof.
		\end{proof}

 \section{Global well-posedness: Proof of Theorem \ref{thm:nlu gwp}}\label{sect:gwp}
 
 In this part, we combine the Fourier truncation method  in \cite{Bourgain98:FTM, KenigPV00:NLW},  the global-in-time (endpoint) Strichartz esitmates in Theorem \ref{thm:St est} and Theorem \ref{thm:endp est} with the energy method to prove the low regularity global well-posedness of \eqref{eq:nlu} in Theorem \ref{thm:nlu gwp}, which is the similar as those in \cite{GallPlan03:NLW}. 

Let $\frac12 <  s\leq 1$ , $(u_0, u_1)\in \left( \dot H^{s}_{D, rad}(\Omega) \cap L^4(\Omega) \right)\times  \dot H^{s-1}_{D, rad}(\Omega)  $. 

\subsection{Global analysis for High frequency part}\label{subs:w H1/2 GWP}

Let us consider the following Dirichlet boundary value problem of  nonlinear wave equation outside of the unit ball with high-frequency-localized radial data.
 \begin{equation}\label{eq:nlw}
 \begin{cases}
 \partial^2_t w -\Delta w + w^3=0, & (t,x)\in\R \times\Omega, \\
 w(0,x)=w_0(x)=P^{\Omega}_{> 2^J}u_0(x),  & x\in \Omega, \\
 \partial_t w(0, x)= w_1(x)=P^{\Omega}_{> 2^J}u_1(x), & x\in \Omega, \\
 w(t,x)=0, & x\in \partial \Omega, 
 \end{cases}
 \end{equation}
 Notice that the distorted Fourier transform preserve the zero Dirichlet boundary value structure. 
 
 Let $\epsilon>0$ sufficiently small, and choose the dyadic number $J =J(\epsilon) \gg 1$ such that
 \begin{align}\label{data:w small}
 \left\| (w_0, w_1)\right\|_{\dot H^{s}_{D,rad}(\Omega) \times \dot H^{s-1}_{D, rad}(\Omega)}   =  \left\| \big (P^{\Omega}_{> 2^J}u_0, P^{\Omega}_{> 2^J}u_1 \big )\right\|_{\dot H^{s}_{D, rad}(\Omega) \times \dot H^{s-1}_{D, rad}(\Omega)}  \lesssim \epsilon . 
 \end{align}

 By the Strichartz estimates in Theorem \ref{thm:St est} and Theorem \ref{thm:endp est}, the standard well-posedness theory  in \cite{LindSog95:NLW, Sogge95:book} together with the regularity theory implies that 
 \begin{proposition}\label{prop:w:gwp} 
 	Let $0<\epsilon\ll 1$ and  $\frac12 <  s\leq 1$,
 there exists a large constant $J=J(\epsilon) >0$ such that if 
 \begin{equation}\label{est:small critnorm}
 2^{J(\frac12 -s)}   \lesssim \epsilon,   
 \end{equation}
 then	\eqref{eq:nlw} is global well-posedness in $ \dot H^{\frac12}_{D}(\Omega) \cap \dot H^{s}_{D}(\Omega)  $. Moreover, we have the following estimates.
 	\begin{align}
 	\|w\|_{L^4_{t,x}(\R\times \Omega)} + \|w\|_{L^{\infty}_tL^3_x(\R\times \Omega)} +
 	 	\|w\|_{L^2_tL^6_x(\R\times \Omega)}  &  \lesssim   2^{J(\frac12-s)}, \label{w:est26} 	\end{align}
 	 	and
 	 	\begin{align}
 	 		\|w\|_{L^{\infty}_t\left(\R; \, \dot H^{s}_{D, rad}(\Omega) \right)}  &  \lesssim \epsilon.\label{w:est s level}
          \end{align}
 	\end{proposition}
 
 \begin{proof}The result follows from the standard contraction mapping argument. More precisely, by the Strichartz estimates in Theorem \ref{thm:St est}, we will show the following map $w\mapsto 	\mathcal{T}(w) $ defined by
	\begin{align*}
 	\mathcal{T}(w) := \cos(t \sqrt{-\Delta_{\Omega}}) w_0 + \frac{\sin(t\sqrt{-\Delta_{\Omega}}) }{\sqrt{-\Delta_{\Omega}}} w_1  -\int^t_0  \frac{\sin((t-s)\sqrt{-\Delta_{\Omega}}) }{\sqrt{-\Delta_{\Omega}}} w^3(s)\; ds 
 	\end{align*}
 	is a contraction on the set $X\subset C\left(\R; \dot H^{\frac12}_{D, rad}(\Omega) \right)  $
 	\begin{align*}
X:= \Big\{ w\in C\left(\R; \dot H^{\frac12}_{D, rad}(\Omega) \right)&  \cap L^4(\R\times \Omega): 	\\
& \|w\|_{L^{\infty}_t\big( \R; \dot H^{\frac12}_{D, rad}(\Omega)\big)\cap L^4_{t,x}(\R\times \Omega)} \leq 2 \, C \left\| \big( w_0, w_1 \big )\right\|_{\dot H^{1/2}_{D, rad}(\Omega) \times \dot H^{-1/2}_{D, rad}(\Omega) } \Big\}	
 	\end{align*}
 	under the partial metric given by
$
 	d(w_1,w_2):=\|w_1-w_2\|_{ L^4_{t,x}(\R\times \Omega)}.
$
 Here $C$ denotes the constant from the Strichartz inequality. By \eqref{est:small critnorm} and the Bernstein estimate in Lemma \ref{lem:berns ests}, we have 
 \begin{align*}
  \left\| \big( w_0, w_1 \big )\right\|_{\dot H^{1/2}_{D, rad}(\Omega) \times \dot H^{-1/2}_{D, rad}(\Omega) }  \lesssim  2^{J(\frac12 -s)}   \lesssim \epsilon. 
 \end{align*}
 
 Using the Strichartz estimate and H\"older inequality, we have
 	\begin{align*}
 		 \|\mathcal{T}(w)\|_{L^{\infty}_t\left(\R; \dot H^{\frac12}_{D, rad}(\Omega) \right)\cap L^4_{t,x}(\R\times \Omega)}    
 		  \lesssim & \left\| \big( w_0, w_1 \big )\right\|_{\dot H^{1/2}_{D, rad}(\Omega) \times \dot H^{-1/2}_{D, rad}(\Omega)}  + \big\|w^3\big\|_{L^{4/3}_{t,x}(\R\times \Omega)} \\
 		 \lesssim & \left\| \big( w_0, w_1 \big )\right\|_{\dot H^{1/2}_{D, rad}(\Omega) \times \dot H^{-1/2}_{D, rad}(\Omega)}  + \big\|w\big\|^3_{L^{4}_{t,x}(\R\times \Omega)}\\
 		 \leq & \; 2 \, C \left\| \big( w_0, w_1 \big )\right\|_{\dot H^{1/2}_{D, rad}(\Omega) \times \dot H^{-1/2}_{D, rad}(\Omega) }. 
 	\end{align*}
 	Arguing as above, we obtain
 		\begin{align*}
 	\|\mathcal{T}(w_1)-\mathcal{T}(w_2)\|_{L^4_{t,x}(\R\times \Omega)}    \lesssim & \big\|w_1^3-w_2^3\big\|_{L^{4/3}_{t,x}(\R\times \Omega)} \\
 	\lesssim &  \big\|w_1-w_2\big\|_{L^{4}_{t,x}(\R\times \Omega)} \left(\big\|w_1\big\|^2_{L^{4}_{t,x}(\R\times \Omega)} + \big\|w_2\big\|^2_{L^{4}_{t,x}(\R\times \Omega)} \right).
 	\end{align*}
 	Thus, choosing $J=J(\epsilon)$ even larger (if necessary), we can guarantee that $\mathcal{T}$ maps the set $X$ back to itself and is a contraction on the set $X$. By the contraction mapping theorem, it follows that $\mathcal{T}$ has a fixed point $w$ in $X$. In addition, by the endpoint Strichartz estimate in Theorem \ref{thm:endp est}, the Sobolev embedding that $L^{3/2}(\Omega)\subset \dot H^{-1/2}_{D}(\Omega)$ and the product rule in Remark \ref{rem:Frac prod}, we have
 		\begin{align*}
 	\|w\|_{L^2_tL^6_{x}(\R\times \Omega)}   \lesssim & \left\| \big( w_0, w_1 \big )\right\|_{\dot H^{1/2}_{D, rad}(\Omega) \times \dot H^{-1/2}_{D, rad}(\Omega)}  + \big\|w^3\big\|_{L^1_t\left(\R;\, \dot H^{-1/2}_{D, rad}(\Omega)\right)} \\
 	\lesssim & \left\| \big( w_0, w_1 \big )\right\|_{\dot H^{1/2}_{D, rad}(\Omega) \times \dot H^{-1/2}_{D, rad}(\Omega)}  + \big\|w^3\big\|_{L^1_t\left(\R;\, L^{3/2}_x(\Omega)\right)} \\
 		\lesssim & \left\| \big( w_0, w_1 \big )\right\|_{\dot H^{1/2}_{D, rad}(\Omega) \times \dot H^{-1/2}_{D, rad}(\Omega)}  + \big\|w\big\|^2_{L^4_{t,x}(\R\times \Omega)}\big\|w\big\|_{L^2_t\left(\R; \, L^{6}_x(\Omega)\right)}\\
 		\lesssim & \left\| \big( w_0, w_1 \big )\right\|_{\dot H^{1/2}_{D, rad}(\Omega) \times \dot H^{-1/2}_{D, rad}(\Omega)}\\
 		\lesssim &  2^{J(\frac12-s)}, 
 	\end{align*}
 	and 
 		\begin{align*}
 	\|w\|_{L^{\infty}_t\left(\R; \dot H^{s}_{D, rad}(\Omega) \right)  \cap L^4_t\dot H^{s-1/2, 4}_{D, rad}(\Omega)} 
 	\lesssim & \, \left\| \big( w_0, w_1 \big )\right\|_{\dot H^{s}_{D, rad}(\Omega) \times \dot H^{s-1}_{D, rad}(\Omega)}   + \big\|w^3\big\|_{L^{4/3}_t\left(\R; \, \dot H^{s-1/2, 4/3}_{D, rad}(\Omega)\right)} \\
 	\lesssim& \, \left\| \big( w_0, w_1 \big )\right\|_{\dot H^{s}_{D, rad}(\Omega) \times \dot H^{s-1}_{D, rad}(\Omega)}   + \big\|w\big\|^2_{L^4_{t,x}(\R\times \Omega)}\big\|w\big\|_{L^4_t\left(\R;\, \dot H^{s-1/2, 4}_{D, rad}(\Omega)\right)}\\
 	\lesssim&  \,  \left\| \big( w_0, w_1 \big )\right\|_{\dot H^{s}_{D, rad}(\Omega) \times \dot H^{s-1}_{D, rad}(\Omega)} \\
 	 \lesssim &  \epsilon.
 	\end{align*}
 This completes the proof.
 	\end{proof}
 
 
%
%

\subsection{Local-in-time energy analysis for low frequency part}\label{subs:v H1 LWP}
Let $w$ be the small solution of \eqref{eq:nlw} in Proposition \ref{prop:w:gwp},   we now consider the following difference equation in the energy space $\left( \dot H^{1}_{D, rad}(\Omega) \cap L^4(\Omega) \right)\times  L^2_{rad}(\Omega) $. 

 \begin{equation}\label{eq:nlv}
\begin{cases}
\partial^2_t v -\Delta v + v^3 =F(v, w), & (t,x)\in\R \times\Omega, \\
v(0,x)=v_0(x)=P^{\Omega}_{\leq 2^J }u_0(x),  & x\in \Omega, \\
\partial_t v(0, x)= v_1(x)=P^{\Omega}_{\leq 2^J }u_1(x), & x\in \Omega, \\
v(t,x)=0, & x\in \partial \Omega, 
\end{cases}
\end{equation}
where 
$
F(v, w)=- 3v^2 w - 3v w^2. 
$

 \begin{proposition}\label{prop:v:lwp}  Let $w$ be the solution in Proposition \ref{prop:w:gwp} and 
 	 $T=T\big(\|v_0\|_{\dot H^1_{D, rad}(\Omega)}, \|v_1\|_{L^2(\Omega)}\big) $ such that  
 	 \begin{align*}
 	  T \cdot   \left\| \big( v_0, v_1 \big )\right\|^2 _{\dot H^{1}_{D, rad}(\Omega)  \times L^2(\Omega)}  \lesssim 1,
 	 \end{align*}
 	 then there exists a unique solution $v \in C\left([0, T); \dot H^1_{D, rad}(\Omega)\right)$ to \eqref{eq:nlv}. Moreover, we have 
 	 	\begin{align*}
 	 		\|v\|_{L^{\infty}\big( [0, T); \dot H^{1}_{D, rad}(\Omega)\big) } \lesssim \left\| \big( v_0, v_1 \big )\right\|_{\dot H^{1}_{D, rad}(\Omega) \times L^2(\Omega) } .
 	 	\end{align*}
 	\end{proposition}
 
 \begin{proof}  The result also follows from the standard contraction mapping argument. More precisely, using the Strichartz estimates from Theorem \ref{thm:St est}, we will show the following map $v\mapsto 	\mathcal{T}(v) $ defined by
 	\begin{align*}
 	\mathcal{T}(v) := \cos(t \sqrt{-\Delta_{\Omega}}) v_0 &+ \frac{\sin(t\sqrt{-\Delta_{\Omega}}) }{\sqrt{-\Delta_{\Omega}}} v_1 \\
 	& -\int^t_0  \frac{\sin((t-s)\sqrt{-\Delta_{\Omega}}) }{\sqrt{-\Delta_{\Omega}}} \left(v^3+3v^2w +3vw^2\right)(s)\; ds 
 	\end{align*}
 	is a contraction on the set $X \subset  C\left([0, T); \dot H^{1}_{D}(\Omega) \right)$:
 	\begin{align*}
 \left\{ v\in C\left([0, T); \dot H^{1}_{D, rad}(\Omega) \right): 	\|v\|_{L^{\infty}\big( [0, T); \dot H^{1}_{D, rad}(\Omega)\big) }\leq 2 \, C \left\| \big( v_0, v_1 \big )\right\|_{\dot H^{1}_{D, rad}(\Omega) \times L^2(\Omega) } \right\},
 	\end{align*}
 where $T$ is determined later and  the metric on $X$ is given by
 	\begin{align*}
 	d(u,v):=\|u-v\|_{ L^{\infty}\big( [0, T); \dot H^{1}_{D, rad}(\Omega)\big)}.
 	\end{align*} 
 	Here $C$ denotes the constant from the Strichartz inequality. 
 	
 		Using the Strichartz estimate and H\"older inequality,  we have
 	
 	\begin{align*}
 & \;	\|\mathcal{T}(v)\|_{L^{\infty}_t\left([0, T); \dot H^{1}_{D, rad}(\Omega) \right)}    \\
 	\lesssim & \left\| \big( v_0, v_1 \big )\right\|_{\dot H^{1}_{D, rad}(\Omega) \times L^2 (\Omega)}  + \big\|v^3+3v^2w +3vw^2\big\|_{L^{1}_tL^2_x([0, T)\times \Omega)} \\
 	\lesssim & \left\| \big( v_0, v_1 \big )\right\|_{\dot H^{1}_{D, rad}(\Omega) \times L^2(\Omega)}  + T\, \big\|v\big\|^3_{L^{\infty}_tL^6_x([0, T)\times \Omega)}    +  \big\|w\|^2_{L^2_{t}L^6_x(\R\times  \Omega)}  \big\|v\big\|_{L^{\infty}_tL^6_x([0, T)\times \Omega)} \\
 	\lesssim & \left\| \big( v_0, v_1 \big )\right\|_{\dot H^{1}_{D, rad}(\Omega) \times L^2(\Omega)}  + T \, \big\|v\big\|^3_{L^{\infty}_t\left([0, T); \dot H^1_{D}( \Omega)\right)}    +  \big\|w\|^2_{L^2_{t}L^6_x(\R\times \Omega)}  \big\|v\big\|_{L^{\infty}_t\left([0, T); \dot H^1_{D, rad}( \Omega)\right)}. 
 	\end{align*}
 	Arguing as above, we obtain
 	\begin{align*}	\|\mathcal{T}(v_1)-\mathcal{T}(v_2)\|&_{L^{\infty}_t\left([0, T); \dot H^{1}_{D, rad}(\Omega) \right)}   \\ \lesssim  &\;  \big\|v_1^3-v_2^3+3v_1^2w -3v_2^2w+3v_1w^2-3v_2w^2\big\|_{L^{1}_tL^2_x([0, T)\times \Omega)} \\
 	\lesssim & \;  T \big\|v_1-v_2\big\|_{L^{\infty}_t\left([0, T); \, L^6_x( \Omega)\right)}  \left(\big\|v_1\big\|^2_{L^{\infty}_t\left([0, T); \, L^6_x( \Omega)\right)} + \big\|v_2\big\|^2_{L^{\infty}_t\left([0, T); \, L^6_x( \Omega)\right)}\right) \\
 	& \;   +  \big\|w\|^2_{L^2_{t}L^6_x( \R\times \Omega)}  \big\|v_1-v_2\big\|_{L^{\infty}_t\left([0, T); \, L^6_x( \Omega)\right)} \\
 	\lesssim &  \;  T \big\|v_1-v_2\big\|_{L^{\infty}_t\dot H^1_{D, rad}( \Omega)}  \left(\big\|v_1\big\|^2_{L^{\infty}_t\dot H^1_{D, rad}( \Omega)} + \big\|v_2\big\|^2_{L^{\infty}_t\dot H^1_{D, rad}( \Omega)}\right)   \\
 	&\;  +  \big\|w\|^2_{L^2_{t}L^6_x( \R\times \Omega)}  \big\|v_1-v_2\big\|_{L^{\infty}_t\dot H^1_{D, rad}( \Omega)}. 
 	\end{align*}
 	Thus, by  Proposition \ref{prop:w:gwp} and choosing $T$ such that $$2\; T \cdot  \left(2 \; C \left\| \big( v_0, v_1 \big )\right\|_{\dot H^{1}_{D, rad}(\Omega)  \times L^2(\Omega)}   \right)^2 \leq \frac12, $$  we can guarantee that $\mathcal{T}$ maps the set $X$ back to itself and is a contraction on the set $X$. By the contraction mapping theorem, it follows that $\mathcal{T}$ has a fixed point $v$ in $X$. 
 	\end{proof}

 \subsection{Global-in-time energy analysis for low frequency part} \label{subs:v H1 est}
 In this part, we extend local solution $v$ of \eqref{eq:nlv} in Proposition \ref{prop:v:lwp} to global one in the energy space $\dot H^1_{D, rad}(\Omega) \times L^2_{rad}(\Omega)$.  By Proposition \ref{prop:v:lwp},   it suffices to control the growth of the energy of  the solution $v$ of \eqref{eq:nlv}, which is not conserved because of the  perturbation term $F(v, w)$. 
 
 Let us denote the energy of $v$ by
 \begin{align*}
 E(v)(t)=  \int_{\Omega} \frac12 \left(\partial_t v(t)\right)^2 + \frac12 \left|\nabla v(t)\right|^2  + \frac14 \left| v(t)\right|^4\; dx,
 \end{align*}
 and take 
$\displaystyle 
 E_T = \sup_{0\leq t<T} E(v)(t), 
$
 where $[0, T)$ is the maximal lifespan interval in Proposition \ref{prop:v:lwp}.

 We now turn to control the energy growth of $v$. 
 
\begin{proposition}\label{prop:v:energy}
Let $w$ be the global solution of \eqref{eq:nlw} in Proposition \ref{prop:w:gwp}, 	 $v$ be the local energy solution of \eqref{eq:nlv} in Proposition \ref{prop:v:lwp}, and $[0, T)$ be the maximal lifespan interval of $v$, then  for any $t\in [0, T)$, we have
	\begin{align*}
	E(v)(t)  \leq E_T \lesssim E(v)(0)  + E_T^{3/2} T^{1/2} \|w\|_{L^2_tL^6_x(\R\times \Omega)}+ E_T \|w\|^2_{L^2_tL^6_x(\R\times \Omega)}.
	\end{align*}
\end{proposition}
\begin{proof}
We now take the derivative in time and obtain that
\begin{align} \label{est:env dt}
\frac{d}{dt} E(v)(t) = \int_{\Omega} \partial_t v \big(-3v^2 w -3 vw^2\big) \; dx. 
\end{align}	
By the H\"older inequality, we obtain that 
\begin{align}
\left| \int^T_0\int_{\Omega} \partial_t v\, v^2 \, w \; dxdt \right|\lesssim & \int^T_0  \|\partial_t v\|_{L^2_x(\Omega)} \|v\|^2_{L^6_x(\Omega)} \|w\|_{L^6_x(\Omega)} \, dt \nonumber\\
 \lesssim &   \int^T_0 \|\partial_t v\|_{L^2_x(\Omega)} \|v\|^2_{\dot H^1_{D}(\Omega)} \|w\|_{L^6_x(\Omega)} \; dt 
\lesssim   E_T^{3/2} \, T^{1/2}\,  \|w\|_{L^2_tL^6_x(\R\times \Omega)}, \label{est:env dt1}
\end{align}
and 
\begin{align}
\left| \int^T_0\int_{\Omega} \partial_t v\, v \, w^2 \; dxdt \right|\lesssim & \int^T_0  \|\partial_t v\|_{L^2_x(\Omega)} \|v\|_{L^6_x(\Omega)} \|w\|^2_{L^6_x(\Omega)} \, dt \nonumber  \\ 
\lesssim &   \int^T_0 \|\partial_t v\|_{L^2_x(\Omega)} \|v\|_{\dot H^1_{D}(\Omega)} \|w\|^2_{L^6_x(\Omega)} \; dt 
\lesssim   E_T \,   \|w\|^2_{L^2_tL^6_x(\R\times \Omega)}. \label{est:env dt2}
\end{align}
Taking \eqref{est:env dt1} and \eqref{est:env dt2} into \eqref{est:env dt}, we obtain the result.
	\end{proof}

 By the fact that $(u_0, u_1)\in \left(\dot H^{s}_{D}(\Omega) \cap L^4(\Omega) \right)\times \dot H^{s-1}_{D}(\Omega) $, Theorem \ref{thm:Sob eq} and  Lemma \ref{lem:berns ests}, we have
 \begin{align*}
 E(v)(0) 
 =& \;  \int_{\Omega} \frac{1}{2} \left| \nabla  v(0)\right|^2 + \frac{1}{2} \left| \partial_t v(0)\right|^2  +\frac{1}{4} \left| v(0)\right|^4 \; dx
 \\
\approx & \; \frac{1}{2} \left\| \left(-\Delta_{\Omega}\right)^{1/2} P^{\Omega}_{\leq 2^J} u_0\right\|^2_{L^2(\Omega)} + \frac{1}{2} \left\|P^{\Omega}_{\leq 2^J}  u_1\right\|^2_{L^2(\Omega)} +\frac{1}{4} \left\| P^{\Omega}_{\leq 2^J}  u_0\right\|^4_{L^4(\Omega)}\\
 \lesssim &\;  2^{2J(1-s)}  \left( \left\| \left(-\Delta_{\Omega}\right)^{s/2}  u_0\right\|^2_{L^2(\Omega)} +  \left\|\left(-\Delta_{\Omega}\right)^{(s-1)/2} u_1\right\|^2_{L^2(\Omega)}  \right) +  \left\|  u_0\right\|^4_{L^4(\Omega)}\\
  \lesssim &\;  2^{2J(1-s)},
 \end{align*}
 then by Proposition \ref{prop:w:gwp} and Proposition \ref{prop:v:energy}, we have 
 	\begin{align*}
 E(v)(t)  \leq E_T \lesssim 2^{2J(1-s)}  + E_T^{3/2} \, T^{1/2} \, 2^{J(1/2-s)}+ E_T \, 2^{2J(1/2 -s)}.
 \end{align*}
For any arbitrarily large time $T$, we can choose $J$ such that
 \begin{align}\label{est:T}
 2^{3J(1-s)} \, T^{1/2} \, 2^{J(1/2-s)} \approx  2^{2J(1-s)}  \Longleftrightarrow T \approx 2^{2J(2s -3/2)}, 
 \end{align}
 we obtain 
 	\begin{align*}
E_T \lesssim 2^{2J(1-s)}  \approx T^{\frac{1-s}{2s-3/2}}.
\end{align*}
which gives  control of the energy growth of $v$ for arbitrarily large $T$ as long as $s>\frac{3}{4}$. 
 
 \subsection{Growth estimate of the solution $u$ of \eqref{eq:nlu} in $\dot H^s_{D, rad}(\Omega) $}\label{subs:u Hs est}
 
From Subsection \ref{subs:w H1/2 GWP} and Subsection \ref{subs:v H1 est}, we know that the solution $u$ of \eqref{eq:nlu} exists in $[0, T)$ for arbitrarily large $T$ in $\dot H^{\frac12}_{D, rad}(\Omega) \cap \dot H^s_{D, rad}(\Omega) + \dot H^{1}_{D, rad}(\Omega) $ with $s>\frac34$. 
 In this part, we show the estimate of the solution $u$ in $\dot H^s_{D, rad}(\Omega) $, and complete the proof of Theorem \ref{thm:nlu gwp}. 

 By Proposition \ref{prop:w:gwp}, it suffices to show the estimate of $v$ in $\dot H^s_{D, rad}(\Omega)$, in addition, the homogeneous part $$\cos(t\sqrt{-\Delta_\Omega})v_0 + \frac{\sin(t\sqrt{-\Delta_\Omega})}{\sqrt{-\Delta_\Omega}}v_1$$ of $v$ is bounded in $\dot H^s_{D}(\Omega)$ by the energy estimate.  By Proposition \ref{prop:v:energy} and the interpolation argument, it reduces to estimate the inhomogeneous part of $v$ in $L^2(\Omega)$.  By the distorted Fourier transform, Proposition \ref{prop:w:gwp} and Proposition \ref{prop:v:energy}, we have for any $0<t<T$
 \begin{align*}
& \left \|  \int^t_0  \frac{\sin\big((t-s)\sqrt{-\Delta_\Omega}\big)}{\sqrt{-\Delta_\Omega}} \big(v^3 + 3v^2w +3vw^2\big)\; ds  \right\|_{L^2(\Omega)} \\
\lesssim   & \int^t_0  \left \| \frac{\sin\big((t-s)\lambda\big)}{\lambda} \mathcal{F}_{D}\big(v^3 + 3v^2w +3vw^2\big)(\lambda) \right\|_{L^2_{\lambda}(\R^+)}  \; ds  \\
\lesssim   & \int^t_0  (t-s) \left \|   v^3 + 3v^2w +3vw^2  \right\|_{L^2_x(\Omega)}  \; ds  \\
\lesssim   & \int^t_0  (t-s)  \left(  \|   v \|^3_{\dot H^1_{D, rad}(\Omega)}  +  \|   v \|^2_{\dot H^1_{D, rad}(\Omega)}   \|   w \|_{L^6_x(\Omega)}   +  \|   v \|_{\dot H^1_{D, rad}(\Omega)}   \|   w \|^2_{L^6_x(\Omega)}   \right) \; ds  \\
\lesssim   &\;  E^{3/2}_T \, \int^t_0  (t-s) \; ds +  E_T\,   \|   w \|_{L^2_tL^6_x(\R\times \Omega)}   \left( \int^t_0  (t-s)^2 \; ds \right)^{1/2} +  E^{1/2}_T \, T\,  \|   w \|^2_{L^2_tL^6_x(\R\times \Omega)}       \\
\lesssim   &\;  E^{3/2}_T \,T^2 +  E_T\,  T^{3/2}\,   \|   w \|_{L^2_tL^6_x(\R\times \Omega)}    +  E^{1/2}_T \, T\,  \|   w \|^2_{L^2_tL^6_x(\R\times \Omega)}       \\
\lesssim & \; T^{2+ \frac{3(1-s)}{4s-3}}, 
 \end{align*}
where $T$ is determined by \eqref{est:T}.
Therefore, by the interpolation between $\dot H^1_{D}(\Omega)$ and $L^2(\Omega)$, we obtain
\begin{align*}
 \left \|\int^t_0  \frac{\sin\big((t-s)\sqrt{-\Delta_\Omega}\big)}{\sqrt{-\Delta_\Omega}} \big(v^3 + 3v^2w +3vw^2\big)\; ds  \right\|_{\dot H^s_{D}(\Omega)} \lesssim  \left( T^{2+ \frac{3(1-s)}{4s-3}} \right)^{1-s} \, \left(T^{\frac{1-s}{4s-3}} \right)^{s}  \lesssim T^{\frac{3(1-s)(2s-1)}{4s-3}}. 
\end{align*}
This completes the proof of Theorem \ref{thm:nlu gwp}.

\appendix 
\section{the integral formula about the half-wave operator}\label{app:A}
	In this appendix, we show the integral formula about the half-wave operator in the radial case in the whole space $\R^3$.  We first recall  the usual Littlewood-Paley theory adapted to the Laplacian operator $-\Delta_{\R^3}$.  Let $\phi$,  $\phi_N$ and $\psi_N$ be defined  by \eqref{def:phi} and \eqref{def:phi scal}, and  $f\in C^{\infty}_{c}(\R^3)$ be a smooth radial function, we define the Littlewood-Paley projections:
	\begin{align*}
	P_{\leq N}f : = \phi_N(\sqrt{-\Delta_{\R^3}})f , \quad  P_{N}f:= \psi_N(\sqrt{-\Delta_{\R^3}})f, \quad  P_{>N}f:= I - P_{\leq N} f, 
	\end{align*}
	and
	\begin{align*}
	\tilde{P}_{N}f:= &\; \tilde{\psi}_N\left(\sqrt{-\Delta_{\R^3}}\right)f  \\
	= &\; \psi_{N-1}\left(\sqrt{-\Delta_{\R^3}}\right)f + \psi_{N}\left(\sqrt{-\Delta_{\R^3}}\right)f +\psi_{N+1}\left(\sqrt{-\Delta_{\R^3}}\right)f. 
	\end{align*}

	Let $r=|x|$, $s=|y|$,  $f\in C^{\infty}_{c}(\R^3)$ be a smooth radial function, then 
		\begin{align}\label{iden:half wave form}
		e^{it{\sqrt{-\Delta_{rad}}}}\tilde{P}_{1}f(x)=\text{Const}\cdot \int_{0}^{+\infty} \int_{0}^{+\infty}\frac{\sin\lambda s}{s}\cdot \frac{\sin\lambda r}{r}\cdot e^{it\cdot \lambda}\cdot \tilde{\psi}_1(\lambda) \,d\lambda \cdot  f(s) s^2 \; ds.  
		\end{align}
		In particular, we have 
			\begin{align*} 
		e^{it{\sqrt{-\Delta_{rad}}}}\tilde{P}_1f(r)= \int_{0}^{+\infty}   \left(	e^{it{\sqrt{-\Delta_{rad}}}}\tilde{P}_1\right)(r;s) \cdot  f(s) s^2 \; ds,   
		\end{align*}
		where the  kernel $	\left(e^{it{\sqrt{-\Delta_{rad}}}}\tilde{P}_1\right)(r;s) $ is defined by
		\begin{align}\label{def:half wave kernel}
		\left(e^{it{\sqrt{-\Delta_{rad}}}}\tilde{P}_1\right)(r;s) := \text{Const}\cdot  \int_{0}^{+\infty}\frac{\sin\lambda s}{s}\cdot \frac{\sin\lambda r}{r}\cdot e^{it\cdot \lambda} \cdot \tilde{\psi}_1(\lambda) \,d\lambda.  
		\end{align}
 
	\begin{proof}[Proof of $(A.1)$]
		By the sphere coordinate $x=r\cdot \theta$, $y=s\cdot \tilde{\theta}$,  $\xi=\lambda \cdot \omega$ and  the Fourier transform in $\R^3$, we have 
		\begin{align*}
		\mathcal{F}f(\xi)= & \; C\cdot \int^{\infty}_{0}\left(\int_{S^2} e^{i s \tilde{\theta}\cdot \lambda\omega} d\sigma(\tilde{\theta}) \right)   \cdot  f(s) \cdot  s^2 \; ds \\
		=& \;  C\cdot \int^{\infty}_{0}\frac{1}{|\lambda s|^{1/2}} J_{1/2}(\lambda s)  \cdot f(s) \cdot  s^2 \; ds,
		\end{align*}
		and 
		\begin{align*}
			e^{it{\sqrt{-\Delta_{rad}}}}\tilde{P}_1f(x)= & \; C\cdot \int^{\infty}_{0}\left(\int_{S^2} e^{i r \theta\cdot \lambda\omega} d\sigma(\omega) \right) \cdot  e^{i\lambda t} \cdot \tilde{\psi}_1(\lambda) \cdot \mathcal{F}f(\lambda) \cdot  \lambda^2 \; d\lambda \\
			=& \;  C\cdot \int^{\infty}_{0}\frac{1}{|\lambda r|^{1/2}} J_{1/2}(\lambda r) \cdot  e^{i\lambda t} \cdot \tilde{\psi}_1(\lambda)\cdot \mathcal{F}f(\lambda) \cdot  \lambda^2 \; d\lambda, 
		\end{align*}
where the constant number $C$ changes line by line,  and $J_{1/2}$ is the Bessel function of $1/2$ order. Since 
$
J_{1/2}(\lambda r)  = 
\sin(\lambda r)/|\lambda r|^{1/2}
$ (we can refer to \cite{Gr14:CFA}), we have
	\begin{align*}
\mathcal{F}f(\lambda)
=& \;  C\cdot \int^{\infty}_{0} \frac{\sin(\lambda s)}{\lambda }    \cdot  f(s) \cdot  s \; ds,
\end{align*}
and 
		\begin{align*}
e^{it{\sqrt{-\Delta_{rad}}}}\tilde{P}_1f(x)
=& \;  C\cdot \int^{\infty}_{0} \frac{\sin(\lambda r)}{r}  \cdot  e^{i\lambda t} \cdot \tilde{\psi}_1(\lambda) \cdot \mathcal{F}f(\lambda) \cdot  \lambda \; d\lambda\\
=& \;  C\cdot \int_{0}^{+\infty} \int_{0}^{+\infty}\frac{\sin\lambda s}{s}\cdot \frac{\sin\lambda r}{r}\cdot e^{it\cdot \lambda} \cdot \tilde{\psi}_1(\lambda) \,d\lambda \cdot  f(s) s^2 \; ds.  
\end{align*}
This completes the proof of \eqref{iden:half wave form}. 
		\end{proof}
	
 \noindent \subsection*{Acknowledgements.}
G. Xu  was supported by National Key Research and Development Program of China (No. 2020YFA0712900) and by NSFC (No. 112371240).  The authors would like to thank the referee for his/her valuable comments and suggestions, which helps to improve our paper.
%
%
%
 

\def\cprime{$'$}

\end{document}